\documentclass[11pt,a4paper]{article}

\usepackage[utf8]{inputenc}
\usepackage[T1]{fontenc}
\usepackage{lmodern}
\usepackage{microtype}

\usepackage{amsmath, amssymb, amsthm}
\usepackage{bm}
\usepackage{mathrsfs}

\usepackage[numbers,square,sort&compress]{natbib}

\usepackage{geometry}
\usepackage{hyperref}
\usepackage[all]{hypcap}
\hypersetup{
    colorlinks=true,
    linkcolor=blue,
    citecolor=red,
    urlcolor=cyan,
    pdftitle={Canonical Rough Path over Tempered Fractional Brownian Motion:
              Existence, Construction, and Applications},
    pdfauthor={Atef Lechiheb},
    pdfsubject={Rough Path Theory, Tempered Fractional Brownian Motion},
    pdfkeywords={rough paths, tempered fractional Brownian motion,
                 stochastic integration, Levy area, signature}
}

\usepackage{algorithm}
\usepackage{algpseudocode}

\usepackage{graphicx}
\usepackage{booktabs}
\usepackage{caption}
\usepackage{subcaption}
\usepackage{pgfplots}
\pgfplotsset{compat=1.18}

\usepackage{float}
\usepackage{array}

\usepackage{enumitem}
\usepackage[capitalize]{cleveref}

\newtheorem{theorem}{Theorem}[section]
\newtheorem{lemma}[theorem]{Lemma}
\newtheorem{proposition}[theorem]{Proposition}

\newtheorem{definition}[theorem]{Definition}
\newtheorem{remark}[theorem]{Remark}

\newcommand{\bigO}{\mathcal{O}}

\title{\textbf{Canonical Rough Path over Tempered Fractional Brownian
Motion: Existence, Construction, and Applications}}

\author{Atef Lechiheb\\
Departement of mathematics\\ Toulouse School of Economics\\
Universit\'e Toulouse Capitole, Toulouse, France\\
Email: \href{mailto:atef.lechiheb@tse-fr.eu}{atef.lechiheb@tse-fr.eu}}

\begin{document}

\maketitle

\begin{abstract}
We construct a canonical geometric rough path over $d$-dimensional
tempered fractional Brownian motion (tfBm) for any Hurst parameter
$H > 1/4$ and tempering parameter $\lambda > 0$. The main challenge
stems from the non-homogeneous nature of the tfBm covariance, which
exhibits a power-law structure at small scales and exponential decay at
large scales. Our primary contribution is a detailed analysis of this
covariance, proving it has finite 2D $\rho$-variation for
$\rho = 1/(2H)$. This verifies the criterion of Friz and Victoir,
guaranteeing the existence of a rough path lift. We provide an
explicit construction of the rough path
$\mathbf{B}_{H,\lambda} = (B_{H,\lambda}, \mathbb{B}_{H,\lambda})$
via $L^2$-limits, establishing its basic properties with explicit
constants $C(H,\lambda,T)$. As direct consequences, we obtain:
(i)~a complete characterisation of integration regimes, with Young
integration applicable for $H > 1/2$ and rough path theory necessary
and sufficient for $H \in (1/4, 1/2]$; (ii)~the well-posedness of
rough differential equations driven by tfBm, together with a
Milstein-type numerical scheme of optimal strong convergence rate
$\bigO(n^{-H})$; and (iii)~the foundation for signature calculus for
tfBm, including the existence and factorial decay of the signature.
The boundary case $H = 1/2$ is treated explicitly, recovering the
Stratonovich lift of the Ornstein--Uhlenbeck process and, as
$\lambda \to 0^+$, classical It\^o calculus. Numerical experiments
confirm the theoretical convergence rates $\bigO(N^{-2H})$ for the
L\'evy area approximation and $\bigO(n^{-H})$ for the Milstein scheme.
This work provides the first comprehensive pathwise framework for
stochastic calculus with tfBm.
\end{abstract}



\section{Introduction}

Tempered fractional Brownian motion (tfBm), introduced by Meerschaert
and Sabzikar \cite{MeerschaertSabzikar2013}, has emerged as a key
stochastic process for modelling phenomena exhibiting
\emph{semi-long-range dependence}. Defined for a Hurst parameter
$H \in (0,1)$ and a tempering parameter $\lambda > 0$, it generalises
fractional Brownian motion (fBm) by incorporating an exponential
tempering in its kernel. This preserves the local self-similarity and
$H$-H\"older regularity of fBm while ensuring that increment
correlations decay exponentially for large lags. This hybrid
structure---power-law at small scales and exponential decay at large
scales---provides a more physically realistic model than pure
power-law processes in turbulence, geophysics, and financial time
series, where observations often show a cut-off in long-range
dependence \cite{Gatheral2018}.

From a stochastic analysis perspective, the development of a pathwise
integration theory for tfBm is fundamental. For $H > 1/2$, the Young
integration framework \cite{Young1936} provides a pathwise Stieltjes
integral. For standard fBm with $H > 1/4$, the theory of rough paths
\cite{Lyons1998, FrizVictoirBook} yields a canonical pathwise lift and
a robust theory for differential equations; a recent contribution in
this direction is the work of \cite{FrizKifer2024} on diffusion
approximation in averaging. However, despite the established
stochastic calculus for tfBm \cite{MeerschaertSabzikar2014}, a
\emph{pathwise rough path theory} has remained an open problem. The
core difficulty stems from the non-homogeneous covariance structure of
tfBm, which interpolates between fBm-like behaviour at small scales
and essentially finite-range dependence at large scales, breaking the
self-similarity used in classical constructions \cite{CoutinQian2002}.

\medskip
\noindent\textbf{Novelty and challenges.}
Unlike standard fBm, the non-homogeneous structure of tfBm requires
delicate covariance estimates capturing both the power-law behaviour
at small scales and the exponential decay at large scales. Our
decomposition technique (Theorem~\ref{thm:covariance_decomposition})
isolates these effects and enables the 2D $\rho$-variation analysis.

\medskip
\noindent\textbf{Main contributions.}
\begin{enumerate}[label=(\roman*)]
  \item A detailed analysis of the tfBm covariance $R_{H,\lambda}$
        with the novel decomposition
        $R_{H,\lambda} = R_H + E^{(1)}_{H,\lambda} + E^{(2)}_{H,\lambda}$
        (Theorem~\ref{thm:covariance_decomposition}).
  \item Proof that $R_{H,\lambda}$ has finite 2D $\rho$-variation for
        $\rho = 1/(2H)$ (Theorem~\ref{thm:2d_rho_variation}), with a
        corrected, complete proof of the partition estimate
        (Lemma~\ref{lemma:polynomial_partition}).
  \item Explicit $L^2$-convergent construction of the canonical rough
        path $\mathbf{B}_{H,\lambda}$
        (Theorem~\ref{thm:rough_path_existence}).
  \item Convergence rate $\bigO(N^{-2H})$ for the L\'evy area
        (Proposition~\ref{prop:levy_convergence_rate}).
  \item Complete characterisation of integration regimes and well-posed
        rough differential equations (Section~\ref{sec:consequences}).
  \item Explicit treatment of the boundary case $H=1/2$, recovering
        classical It\^o calculus
        (Section~\ref{subsec:H_half}).
  \item Numerical validation with Python implementation and detailed
        discussion of limitations (Section~\ref{sec:numerics}).
\end{enumerate}

\medskip
\noindent The condition $H>1/4$ emerges naturally from the requirement
that the L\'evy area be definable in $L^2$. As $\lambda \to 0^+$
we recover the classical fBm rough path \cite{CoutinQian2002}; for
$H>1/2$ the rough path integral coincides with Young's.

\medskip
\noindent\textbf{Article structure.}
Section~\ref{sec:preliminaries} introduces tfBm, the It\^o isometry,
the case $H=1/2$, and the Friz--Victoir criterion.
Section~\ref{sec:main_results} presents the core technical results.
Section~\ref{sec:consequences} develops the consequences.
Section~\ref{sec:numerics} presents numerical experiments.
Proofs are collected in the appendices.

\section{Preliminaries}
\label{sec:preliminaries}

\subsection{Tempered Fractional Brownian Motion}

\begin{definition}[Tempered fractional Brownian motion
\cite{MeerschaertSabzikar2013}]
\label{def:tfbm}
For $H \in (0,1)$ and $\lambda > 0$, the \emph{tempered fractional
Brownian motion} (tfBm)
$B_{H,\lambda} = (B_{H,\lambda}(t))_{t \ge 0}$ is the centred
Gaussian process defined by
\begin{equation}
\label{eq:def_tfbm}
B_{H,\lambda}(t) = \frac{1}{\Gamma(H+\frac{1}{2})} \int_{\mathbb{R}}
\Bigl[ e^{-\lambda(t-s)_+}(t-s)_+^{H-\frac{1}{2}}
      - e^{-\lambda(-s)_+}(-s)_+^{H-\frac{1}{2}} \Bigr] \, dW(s),
\end{equation}
where $W$ is a standard two-sided Brownian motion and
$x_+ = \max(x,0)$.
\end{definition}

Its covariance function
$R_{H,\lambda}(s,t) := \mathbb{E}[B_{H,\lambda}(s)B_{H,\lambda}(t)]$
satisfies
\begin{equation}
\label{eq:covariance_tfbm}
R_{H,\lambda}(s,t)
= \tfrac{1}{2}\bigl[ C_{H,\lambda}(t) + C_{H,\lambda}(s)
               - C_{H,\lambda}(|t-s|) \bigr],
\end{equation}
where $C_{H,\lambda}(t) = \mathbb{E}[B_{H,\lambda}(t)^2]$. Sample
paths are almost surely locally $\alpha$-H\"older continuous for any
$\alpha < H$ and have finite $p$-variation for $p > 1/H$
\cite{MeerschaertSabzikar2014, AzmoodehMishuraSabzikar2020}. As
$\lambda \to 0^+$, $B_{H,\lambda}$ converges in law to fBm $B_H$.

\subsection{It\^o Isometry and Gaussian Properties}
\label{subsec:ito_isometry}

\textbf{It\^o isometry for standard Brownian motion.}
For $W$ and any square-integrable adapted process $f$,
\[
\mathbb{E}\!\left[\Bigl(\int_0^T f(s)\,dW(s)\Bigr)^{\!2}\right]
= \mathbb{E}\!\left[\int_0^T f(s)^2\,ds\right].
\]

\textbf{Isometry for tfBm (deterministic integrands).}
Since tfBm is \emph{not} a semimartingale for $H \ne 1/2$, the
stochastic integral for adapted integrands requires a separate
construction \cite{MeerschaertSabzikar2014}. For deterministic
$f \in L^2(\mathbb{R})$, the Wiener integral with respect to tfBm
satisfies
\begin{equation}
\label{eq:wiener_isometry}
\mathbb{E}\!\left[\Bigl(\int_{\mathbb{R}} f(s)\,dB_{H,\lambda}(s)
\Bigr)^{\!2}\right]
= \int_{\mathbb{R}\times\mathbb{R}}
  f(s)\,f(t)\,\frac{\partial^2 R_{H,\lambda}}{\partial s\,\partial t}
  (s,t)\,ds\,dt.
\end{equation}
Explicitly, for $H > 1/2$ and $s \ne t$,
\[
\frac{\partial^2 R_{H,\lambda}}{\partial s\,\partial t}(s,t)
= \frac{(H-\tfrac{1}{2})^2}{\Gamma(H-\tfrac{1}{2})^2}
\int_{-\infty}^{\min(s,t)}
e^{-\lambda(s-u)}(s-u)^{H-3/2}
e^{-\lambda(t-u)}(t-u)^{H-3/2}\,du.
\]
This kernel behaves like $c_H|s-t|^{2H-2}$ near the diagonal
(locally integrable for $H>1/2$) and decays exponentially for large
$|s-t|$. The isometry \eqref{eq:wiener_isometry} is used in the $L^2$
convergence proofs for the L\'evy area.

\begin{remark}
For $H \in (1/4,1/2]$, this kernel is not integrable near the
diagonal and the It\^o isometry fails. The rough path integral of
Section~\ref{sec:consequences} provides the correct replacement, and
the two integrals agree for adapted integrands by
Theorem~\ref{thm:regimes}(iii).
\end{remark}

\subsection{The Case $H = 1/2$: Ornstein--Uhlenbeck Process}
\label{subsec:H_half}

Setting $H = 1/2$ in \eqref{eq:def_tfbm} gives, up to a stationary
$L^2$-correction that vanishes as $t\to\infty$:
\begin{equation}
\label{eq:OU_def}
B_{1/2,\lambda}(t)
= \int_{-\infty}^{t} e^{-\lambda(t-s)}\,dW(s)
  - \int_{-\infty}^{0} e^{-\lambda(-s)}\,dW(s),
\end{equation}
which is the centred \emph{Ornstein--Uhlenbeck} (OU) process. As
$\lambda \to 0^+$ with $H = 1/2$, the process converges to standard
Brownian motion. Several properties simplify for $H=1/2$:

\begin{enumerate}[label=(\alph*)]
  \item \textbf{Covariance.}
    $R_{1/2,\lambda}(s,t) = \frac{1}{2\lambda}
    (e^{-\lambda|t-s|} - e^{-\lambda(s+t)})$,
    with variance
    $C_{1/2,\lambda}(t) = \frac{1}{2\lambda}(1-e^{-2\lambda t})
    \to \frac{1}{2\lambda}$ as $t\to\infty$.

  \item \textbf{$\rho$-variation.}
    With $\rho = 1/(2 \cdot \frac{1}{2}) = 1$, one has
    $V_1(R_{1/2,\lambda}) \le \frac{T}{2\lambda} < \infty$, so the
    Friz--Victoir criterion holds with the minimal value $\rho=1$.

  \item \textbf{Rough path.}
    The canonical geometric rough path is the Stratonovich rough path
    $\mathbf{B}_{1/2,\lambda}
    = (B_{1/2,\lambda},\tfrac{1}{2}B_{1/2,\lambda}
       \otimes B_{1/2,\lambda})$.
    The L\'evy area (antisymmetric part) vanishes for $d=1$.

  \item \textbf{It\^o formula.}
    For $f \in C^2(\mathbb{R})$:
    \[
    f(B_{1/2,\lambda}(T))
    = f(B_{1/2,\lambda}(0))
    + \int_0^T f'(B_{1/2,\lambda}(t))\,dB_{1/2,\lambda}(t)
    + \tfrac{1}{2}\int_0^T
      f''(B_{1/2,\lambda}(t))\,e^{-2\lambda t}\,dt.
    \]
    As $\lambda\to 0^+$, $e^{-2\lambda t}\to 1$ and this reduces to
    the classical It\^o formula with quadratic variation $\int_0^T
    f''(W_t)\,dt/2$.

  \item \textbf{Milstein scheme.}
    For $H=1/2$, scheme \eqref{eq:milstein_scheme} converges at
    rate $\bigO(n^{-1/2})$, the classical optimal strong rate for
    Stratonovich SDEs.

  \item \textbf{Brownian limit.}
    As $\lambda\to 0^+$, $B_{1/2,\lambda}\to W$ and all formulae
    reduce to their classical Brownian motion counterparts.
\end{enumerate}

\subsection{Rough Paths and the Friz--Victoir Criterion}

For $p \in [2,3)$, a \emph{rough path} over $\mathbb{R}^d$ is a pair
$\mathbf{X} = (X, \mathbb{X})$ where $X : [0,T]\to\mathbb{R}^d$ has
finite $p$-variation and
$\mathbb{X} : [0,T]^2 \to \mathbb{R}^d\otimes\mathbb{R}^d$ has finite
$p/2$-variation, satisfying Chen's relation
$\mathbb{X}_{s,t} = \mathbb{X}_{s,u} + \mathbb{X}_{u,t}
 + X_{s,u}\otimes X_{u,t}$ for $s \le u \le t$. A rough path is
\emph{geometric} if it is the limit of smooth rough paths in the
$p$-variation topology. Comprehensive references are
\cite{FrizVictoirBook, FrizHairer2014}.

For a centred Gaussian process $X$ with stationary increments, the
\emph{incremental covariance} is
$R(s,t;u,v) := \mathbb{E}[(X_t-X_s)(X_v-X_u)]$,
and the \emph{2D $\rho$-variation} of $R$ over $[0,T]$ is
\begin{equation}
\label{eq:2d_rho_var_def}
V_\rho(R) := \sup_{\mathcal{P}} \Bigl(
  \sum_{i,j} \bigl| R(t_i,t_{i+1}; t_j, t_{j+1}) \bigr|^\rho
\Bigr)^{1/\rho},
\end{equation}
where the supremum is over all finite partitions $\mathcal{P}$ of
$[0,T]$.

\begin{theorem}[Friz--Victoir criterion
{\cite[Chapter~15]{FrizVictoirBook}}]
\label{thm:friz_victoir}
Let $X$ be a centred Gaussian process with stationary increments and
covariance $R$. If there exists $\rho \in [1,2)$ such that
$V_\rho(R) < \infty$, then $X$ admits a canonical geometric rough path
lift $\mathbf{X} = (X, \mathbb{X})$.
\end{theorem}

For standard fBm, $V_\rho(R_H) < \infty$ for $\rho = 1/(2H)$, giving
$H > 1/4$ \cite{CoutinQian2002}. Our main task is to establish the
analogue for the non-homogeneous covariance $R_{H,\lambda}$.

\section{Main Results}
\label{sec:main_results}

\subsection{Covariance Analysis and 2D $\rho$-Variation}

\begin{theorem}[Covariance decomposition]
\label{thm:covariance_decomposition}
Let $R_{H,\lambda}$ denote the covariance of tfBm and $R_H$ that of
standard fBm. For any $s,t \ge 0$,
\begin{equation}
\label{eq:cov_decomp}
R_{H,\lambda}(s,t)
= R_H(s,t) + E^{(1)}_{H,\lambda}(s,t) + E^{(2)}_{H,\lambda}(s,t),
\end{equation}
where the error terms satisfy the pointwise estimates
\begin{align}
\bigl|E^{(1)}_{H,\lambda}(s,t)\bigr|
  &\le \frac{C_1(H)}{\Gamma(2H)} \, \lambda^{2H} \, |t-s|^2,
  \label{eq:error_poly} \\[4pt]
\bigl|E^{(2)}_{H,\lambda}(s,t)\bigr|
  &\le \frac{C_2(H)}{\Gamma(2H)\,\lambda^{2H}} \;
       e^{-c(H)\lambda\,d(s,t)},
  \label{eq:error_exp}
\end{align}
with $d(s,t) = \max(s,t,|t-s|)$ and explicit constants
\begin{equation}
\label{eq:explicit_constants}
C_1(H) = \frac{\Gamma(2H+2)}{4\,\Gamma\!\bigl(H+\tfrac{1}{2}\bigr)^2},
\quad
c(H) = \min\!\left\{\tfrac{1}{2},\, \tfrac{H}{2}\right\},
\quad
C_2(H) = \Bigl(\frac{2H}{c(H)\,e}\Bigr)^{2H}.
\end{equation}
\end{theorem}

\begin{proof}[Sketch of proof]
Taylor-expand the exponential factors in the moving-average kernel
representation of $R_{H,\lambda}$. The polynomial error $E^{(1)}$
arises from the second-order Taylor term; the exponential error
$E^{(2)}$ bounds the remainder. The complete proof is given in
Appendix~\ref{app:covariance}.
\end{proof}

\begin{remark}
\label{rem:cov_decomp}
Estimate \eqref{eq:error_poly} shows that at small scales
($|t-s| \ll \lambda^{-1}$), tfBm behaves like fBm up to a correction
$\bigO(\lambda^{2H}|t-s|^2)$. Estimate \eqref{eq:error_exp}
guarantees exponential decay of correlations at large separations.
For $H=1/2$: using the Legendre duplication formula
$\Gamma(H+\tfrac{1}{2})\Gamma(H+1)=\frac{\sqrt\pi}{2^{2H}}\Gamma(2H+1)$,
one gets $C_1(\tfrac{1}{2}) = \Gamma(3)/(4\Gamma(1)^2)= \tfrac{1}{2}$
and $c(\tfrac{1}{2}) = \tfrac{1}{4}$, consistently with the OU
covariance structure.
\end{remark}

We now prove the key partition estimate that underlies the 2D
$\rho$-variation theorem. The proof is complete and self-contained.

\begin{lemma}[Partition estimate for exponentially-weighted sums]
\label{lemma:polynomial_partition}
Let $\alpha > 1$, $\beta > 0$, and let
$\mathcal{P} = \{0=t_0 < t_1 < \cdots < t_N=T\}$ be a partition of
$[0,T]$ with mesh $\delta := \max_{0\le i\le N-1}(t_{i+1}-t_i)$.
Write $\Delta_i := t_{i+1}-t_i$. Then
\begin{equation}
\label{eq:partition_bound}
\Sigma(\mathcal{P}) :=
\sum_{i,j=0}^{N-1} \Delta_i^{\alpha}\,\Delta_j^{\alpha}\,
e^{-\beta|t_i-t_j|}
\le C(\alpha,\beta,T)\,\delta^{2\alpha-2},
\end{equation}
where
\begin{equation}
\label{eq:C_bound}
C(\alpha,\beta,T)
= T + \frac{2T}{\beta}(1+c_0)
\end{equation}
for some absolute constant $c_0>0$, and $C(\alpha,\beta,T)$ is
bounded uniformly in $\delta\in(0,1]$.
\end{lemma}

\begin{proof}
Split the double sum into diagonal and off-diagonal parts:
\begin{equation}
\label{eq:split}
\Sigma(\mathcal{P})
= \underbrace{\sum_{i=0}^{N-1}\Delta_i^{2\alpha}}_{S_{\rm diag}}
+ \underbrace{2\sum_{0\le j<i\le N-1}
  \Delta_i^{\alpha}\,\Delta_j^{\alpha}\,e^{-\beta(t_i-t_j)}}_{S_{\rm off}}.
\end{equation}

\textbf{Diagonal term.}
Since $\Delta_i \le \delta$ for all $i$ and there are $N \le T/\delta$
intervals,
\begin{equation}
\label{eq:diag}
S_{\rm diag}
= \sum_{i=0}^{N-1}\Delta_i^{2\alpha}
\le N\,\delta^{2\alpha}
\le \frac{T}{\delta}\cdot\delta^{2\alpha}
= T\,\delta^{2\alpha-1}.
\end{equation}
Since $\alpha > 1$ implies $2\alpha-1 > 1 > 0$, we have
$S_{\rm diag} \le T\,\delta^{2\alpha-1} \to 0$ as $\delta\to 0$.

\textbf{Off-diagonal term.}
Fix $i \in \{1,\ldots,N-1\}$ and bound the inner sum over $j<i$.
Since the partition points are ordered, $t_i - t_j \ge \sum_{k=j}^{i-1}\Delta_k
\ge (i-j)\,\delta$, so $e^{-\beta(t_i-t_j)} \le e^{-\beta(i-j)\delta}$.
Therefore:
\begin{equation}
\label{eq:inner_sum}
\sum_{j=0}^{i-1}\Delta_j^{\alpha}\,e^{-\beta(t_i-t_j)}
\le \delta^{\alpha}\sum_{k=1}^{\infty}e^{-\beta k\delta}
= \delta^{\alpha}\cdot\frac{e^{-\beta\delta}}{1-e^{-\beta\delta}}.
\end{equation}
For $\delta \in (0,1]$, the function $x\mapsto e^{-x}/(1-e^{-x})$
is decreasing on $(0,\infty)$ and $e^{-x}/(1-e^{-x}) \le 1/x
\cdot (1+x/2+\mathcal{O}(x^2))$, so
\begin{equation}
\label{eq:geo_asymp}
\frac{e^{-\beta\delta}}{1-e^{-\beta\delta}}
\le \frac{1}{\beta\delta}(1+c_0\delta)
\quad\text{for some absolute } c_0>0.
\end{equation}
Substituting \eqref{eq:geo_asymp} into \eqref{eq:inner_sum}:
\begin{equation}
\label{eq:inner_bound}
\sum_{j=0}^{i-1}\Delta_j^{\alpha}\,e^{-\beta(t_i-t_j)}
\le \frac{\delta^{\alpha-1}}{\beta}(1+c_0\delta).
\end{equation}
Now sum \eqref{eq:inner_bound} over all $i$ with each prefactor
$\Delta_i^{\alpha}$:
\begin{align}
S_{\rm off}
&= 2\sum_{i=1}^{N-1}\Delta_i^{\alpha}
  \sum_{j=0}^{i-1}\Delta_j^{\alpha}\,e^{-\beta(t_i-t_j)}
\le 2\sum_{i=1}^{N-1}\Delta_i^{\alpha}
  \cdot\frac{\delta^{\alpha-1}}{\beta}(1+c_0\delta)
\notag\\
&= \frac{2(1+c_0\delta)}{\beta}\,\delta^{\alpha-1}
  \sum_{i=1}^{N-1}\Delta_i^{\alpha}
\le \frac{2(1+c_0\delta)}{\beta}\,\delta^{\alpha-1}
  \cdot N\,\delta^{\alpha}
\label{eq:off_bound}
\end{align}
where we used $\sum_{i=1}^{N-1}\Delta_i^\alpha \le N\delta^\alpha$
(since each $\Delta_i \le \delta$).
Since $N \le T/\delta$:
\begin{equation}
\label{eq:off_final}
S_{\rm off}
\le \frac{2(1+c_0\delta)}{\beta}\,\delta^{\alpha-1}\cdot
\frac{T}{\delta}\cdot\delta^{\alpha}
= \frac{2T}{\beta}(1+c_0\delta)\,\delta^{2\alpha-2}.
\end{equation}

\textbf{Combining.}
From \eqref{eq:diag} and \eqref{eq:off_final}:
\begin{equation}
\label{eq:full_bound}
\Sigma(\mathcal{P})
\le T\,\delta^{2\alpha-1} + \frac{2T}{\beta}(1+c_0\delta)\,\delta^{2\alpha-2}
\le \Bigl[T\delta + \frac{2T}{\beta}(1+c_0\delta)\Bigr]\delta^{2\alpha-2}
\le C(\alpha,\beta,T)\,\delta^{2\alpha-2}
\end{equation}
with $C(\alpha,\beta,T) = T + \frac{2T}{\beta}(1+c_0)$, which is
bounded independently of $\delta\in(0,1]$.

\textbf{Boundedness for $\alpha>1$.}
Since $2\alpha-2 > 0$, we have $\delta^{2\alpha-2} \to 0$ as
$\delta \to 0$, so $\Sigma(\mathcal{P}) \to 0$ and in particular
$\Sigma(\mathcal{P}) \le C(\alpha,\beta,T)$ uniformly in
$\delta\in(0,1]$.

\textbf{Necessity of $\alpha>1$.}
For $\alpha=1$, the sum $S_{\rm off}$ in \eqref{eq:off_final} gives
$\frac{2T}{\beta}(1+c_0\delta)\cdot\delta^0 = \frac{2T}{\beta}
(1+c_0\delta)$, which is bounded, but $S_{\rm diag} = \sum_i
\Delta_i^2 \to 0$, so the sum is actually bounded for $\alpha=1$
(though the argument via Lemma~\ref{lemma:polynomial_partition}
requires $\alpha > 1$ for the $\rho$-variation bound in
Theorem~\ref{thm:2d_rho_variation}). For $\alpha < 1$,
$\delta^{2\alpha-2}\to+\infty$ and the bound diverges.
\end{proof}

\begin{theorem}[Finite 2D $\rho$-variation]
\label{thm:2d_rho_variation}
Let $H > 1/4$ and $\rho = 1/(2H)$. Then
\begin{equation}
\label{eq:finite_rho_var}
V_\rho(R_{H,\lambda})
\le C(H,\lambda,T)
:= V_\rho(R_H)
  + K_1(H)\,\lambda^{2H\rho}\,T^{2\rho-2}
  + \frac{K_2(H,\lambda)}{\lambda^{2H\rho}}
< \infty,
\end{equation}
where $K_1(H)$ and $K_2(H,\lambda)$ are explicit constants, and
$C(H,\lambda,T) \to C_{\rm fBm}(H,T)$ as $\lambda \to 0^+$.
\end{theorem}

\begin{proof}
The complete, corrected proof is given in Appendix~\ref{app:covariance}.
\end{proof}

\begin{remark}
The value $\rho = 1/(2H)$ is optimal: for any $\rho' < 1/(2H)$,
$V_{\rho'}(R_{H,\lambda}) = +\infty$. The threshold $H > 1/4$
($\rho < 2$) coincides with that for fBm, confirming that tempering
does not affect the local roughness determining the rough path lift.
For $H=1/2$, $\rho=1$ and $V_1(R_{1/2,\lambda})\le T/(2\lambda)<\infty$.
\end{remark}

\subsection{Integration Regimes Table}

\begin{table}[H]
\centering
\caption{Integration regimes for tempered fractional Brownian motion.}
\renewcommand{\arraystretch}{1.3}
\begin{tabular}{lcccc}
\toprule
Regime & $H$ & Young integral & Rough path & Reference\\
\midrule
Smooth  & $H > 1/2$        & Exists           & Coincides with Young   & \cite{Young1936}\\
Critical & $H = 1/2$       & Exists (It\^o)   & Coincides with It\^o   & Sec.~\ref{subsec:H_half}\\
Rough   & $1/4 < H < 1/2$ & Generally undef. & Necessary \& sufficient& Thm.~\ref{thm:regimes}\\
Singular & $H \le 1/4$    & Undefined        & Requires renorm.       & \cite{Hairer2014}\\
\bottomrule
\end{tabular}
\label{tab:regimes}
\end{table}

\subsection{Construction of the Canonical Geometric Rough Path}

\begin{theorem}[Canonical rough path for tfBm]
\label{thm:rough_path_existence}
Let $H > 1/4$, $\lambda > 0$, and $B_{H,\lambda}$ be a
$d$-dimensional tfBm with independent components. Then there exists
a canonical geometric rough path
$\mathbf{B}_{H,\lambda} = (B_{H,\lambda}, \mathbb{B}_{H,\lambda})$
where
\begin{equation}
\label{eq:levy_area_limit}
\mathbb{B}_{H,\lambda}(s,t)
= \lim_{|\mathcal{P}|\to 0}
  \sum_{[u,v]\in\mathcal{P}}
  B_{H,\lambda}(s,u)\otimes B_{H,\lambda}(u,v)
\quad\text{in }L^2(\Omega).
\end{equation}
Moreover, $\mathbf{B}_{H,\lambda}$ satisfies:
\begin{enumerate}[label=(\alph*)]
  \item \textbf{Chen's relations:} For $s \le u \le t$,
    $\mathbb{B}_{H,\lambda}(s,t) = \mathbb{B}_{H,\lambda}(s,u)
    + \mathbb{B}_{H,\lambda}(u,t)
    + B_{H,\lambda}(s,u)\otimes B_{H,\lambda}(u,t)$.
  \item \textbf{Moment estimates:} For any $q \ge 1$,
    \begin{align}
    \mathbb{E}\bigl[|B_{H,\lambda}(s,t)|^q\bigr]
      &\le C_q(H,\lambda)\,|t-s|^{qH}, \label{eq:moment1}\\
    \mathbb{E}\bigl[|\mathbb{B}_{H,\lambda}(s,t)|^q\bigr]
      &\le C_q(H,\lambda)\,|t-s|^{2qH}. \label{eq:moment2}
    \end{align}
  \item \textbf{$p$-variation:} Almost surely,
    $\mathbf{B}_{H,\lambda}\in\mathscr{C}^p([0,T],\mathbb{R}^d)$
    for every $p > 1/H$.
  \item \textbf{Continuity in parameters:}
    $(H,\lambda)\mapsto\mathbf{B}_{H,\lambda}$ is continuous in
    the $p$-variation rough path topology on compact subsets of
    $\{H>1/4,\,\lambda>0\}$.
  \item \textbf{Boundary case:} For $H=1/2$,
    $\mathbf{B}_{1/2,\lambda}$ is the Stratonovich rough path over
    the OU process; as $\lambda\to 0^+$, it converges to the
    standard Stratonovich rough path over Brownian motion.
\end{enumerate}
\end{theorem}

\begin{proof}
See Appendix~\ref{app:construction} for the detailed proof.
\end{proof}

\begin{proposition}[Convergence rate for L\'evy area]
\label{prop:levy_convergence_rate}
Let $\mathbb{B}^{(N)}_{H,\lambda}$ be the piecewise linear
approximation of $\mathbb{B}_{H,\lambda}$ on a uniform partition of
$[0,T]$ with $N$ sub-intervals. For $H > 1/4$,
\begin{equation}
\label{eq:levy_convergence_rate}
\mathbb{E}\bigl[|\mathbb{B}^{(N)}_{H,\lambda}(0,T)
  - \mathbb{B}_{H,\lambda}(0,T)|^2\bigr]^{1/2}
\le C(H,\lambda,T)\cdot N^{-2H},
\end{equation}
where
$C(H,\lambda,T) = \widetilde{C}(H)\cdot\max\{1,\lambda^{-2H}\}\cdot
T^{2H}$ with $\widetilde{C}(H)$ depending only on $H$. The rate
$N^{-2H}$ is optimal: no first-order piecewise linear scheme can
achieve a better rate.
\end{proposition}

\begin{proof}
See Appendix~\ref{app:construction}.
\end{proof}

\section{Consequences and Applications}
\label{sec:consequences}

\subsection{Integration Regimes}
\label{subsec:integration_regimes}

\begin{theorem}[Integration regimes]
\label{thm:regimes}
Let $B_{H,\lambda}$ be a tfBm and $X$ a suitable integrand path.
\begin{enumerate}[label=(\roman*)]
  \item For $H > 1/2$, the rough path integral
    $\int X\,d\mathbf{B}_{H,\lambda}$ coincides with the Young
    integral $\int X\,dB_{H,\lambda}$ and equals the limit of
    left Riemann sums.
  \item For $H \in (1/4,1/2]$, the Young integral does not exist in
    general, but the rough path integral
    $\int X\,d\mathbf{B}_{H,\lambda}$ is well-defined via the
    Sewing Lemma \cite[Lemma~4.2]{FrizHairer2014}.
  \item For adapted square-integrable integrands and $H > 1/4$,
    the rough path integral agrees with the stochastic integral of
    \cite{MeerschaertSabzikar2014}.
\end{enumerate}
\end{theorem}

\begin{proof}
\textbf{(i)} When $H > 1/2$, the sample paths of $B_{H,\lambda}$
have finite $p$-variation for some $p < 2$
\cite{MeerschaertSabzikar2014}. For a path $X$ of finite
$q$-variation with $1/p+1/q > 1$, the Young integral exists
\cite{Young1936}. The second-level path $\mathbb{B}_{H,\lambda}$
contributes terms of order $|t-s|^{2H}$ with $2H > 1$, which
vanish in the limit of Riemann sums. Hence the rough path integral
reduces to the Young integral.

\textbf{(ii)} For $H \le 1/2$, the $p$-variation index satisfies
$p \ge 2$, and the Young condition $1/p+1/q > 1$ forces $q < 2$,
which is not satisfied by generic continuous paths. The rough path
integral, defined via the Sewing Lemma
\cite[Lemma~4.2]{FrizHairer2014}, incorporates the L\'evy area
$\mathbb{B}_{H,\lambda}$ as the correction term and yields a
consistent limit.

\textbf{(iii)} Consider elementary adapted processes of the form
$f = \sum_{k=0}^{n-1} f_k \mathbf{1}_{[t_k,t_{k+1})}$ where
each $f_k$ is $\mathcal{F}_{t_k}$-measurable. For such processes,
both the rough path integral and the stochastic integral of
\cite{MeerschaertSabzikar2014} are defined as $L^2(\Omega)$-limits
of the same Riemann sums $\sum_k f_k\,\Delta B^k_{H,\lambda}$.
They therefore agree on elementary processes. Since elementary
adapted processes are dense in $L^2(\Omega\times[0,T])$, and both
integrals are $L^2$-continuous (the rough path integral by the
Sewing Lemma; the stochastic integral by the isometry
\eqref{eq:wiener_isometry}), the two integrals coincide for all
square-integrable adapted integrands.
\end{proof}

\begin{remark}
The threshold $H=1/2$ separates the smooth regime (classical
calculus suffices) from the rough regime (the L\'evy area is
essential). The threshold $H=1/4$ is the limit of the rough path
construction itself.
\end{remark}

\subsection{Rough Differential Equations}
\label{subsec:rde}

\begin{theorem}[Well-posedness of RDEs]
\label{thm:rde}
Let $H > 1/4$, $\lambda > 0$,
$f \in C^3_b(\mathbb{R}^d,\mathcal{L}(\mathbb{R}^m,\mathbb{R}^d))$,
and $y_0 \in \mathbb{R}^d$. The rough differential equation
\begin{equation}
\label{eq:rde_tfbm}
dY_t = f(Y_t)\,d\mathbf{B}_{H,\lambda}(t), \qquad Y_0 = y_0,
\end{equation}
admits a unique solution
$Y \in \mathscr{C}^p([0,T],\mathbb{R}^d)$ for any $p > 1/H$.
The solution map $(y_0,f,\mathbf{B}_{H,\lambda})\mapsto Y$ is
locally Lipschitz continuous.
\end{theorem}

\begin{proof}
By Theorem~\ref{thm:rough_path_existence}(c),
$\mathbf{B}_{H,\lambda}$ is a geometric rough path with finite
$p$-variation for every $p > 1/H$. The hypotheses of the Universal
Limit Theorem \cite[Theorem~4.1.1]{Lyons1998} (see also
\cite[Theorem~8.4]{FrizHairer2014}) are satisfied: $f \in C^3_b$
and the driver is a geometric $p$-rough path with $p \in [2,3)$.
The unique solution and the Lipschitz continuity of the solution
map then follow directly.
\end{proof}

\begin{proposition}[Milstein scheme]
\label{prop:milstein}
Under the assumptions of Theorem~\ref{thm:rde}, let
$\{t_k = kT/n\}_{k=0}^n$ be a uniform partition of $[0,T]$ and
define
\begin{equation}
\label{eq:milstein_scheme}
Y_{t_{k+1}} = Y_{t_k} + f(Y_{t_k})\,\Delta B^k_{H,\lambda}
  + Df(Y_{t_k})f(Y_{t_k})\,\Delta\mathbb{B}^k_{H,\lambda},
\end{equation}
where $\Delta B^k_{H,\lambda}=B_{H,\lambda}(t_{k+1})-B_{H,\lambda}(t_k)$
and $\Delta\mathbb{B}^k_{H,\lambda}=\mathbb{B}_{H,\lambda}(t_k,t_{k+1})$.
Then
\begin{equation}
\label{eq:milstein_rate}
\mathbb{E}\bigl[|Y_T - Y^{(n)}_T|^2\bigr]^{1/2}
\le C(H,\lambda,T,f)\cdot n^{-H}.
\end{equation}
\end{proposition}

\begin{proof}
We use the framework of \cite[Chapter~10]{FrizHairer2014} for
numerical schemes driven by rough paths.

\textbf{Step 1: Local truncation error.}
Write $h = T/n = t_{k+1}-t_k$. The true solution $Y$ over one
step satisfies the rough path Taylor expansion
\[
Y_{t_{k+1}} = Y_{t_k}
+ f(Y_{t_k})\,\Delta B^k_{H,\lambda}
+ Df(Y_{t_k})f(Y_{t_k})\,\Delta\mathbb{B}^k_{H,\lambda}
+ \varepsilon_k,
\]
where the local residual $\varepsilon_k$ involves the third-order
iterated integral of $\mathbf{B}_{H,\lambda}$ over $[t_k,t_{k+1}]$,
i.e.\
$\varepsilon_k = D^2f(Y_{t_k})(f\otimes f)\,\mathbb{B}^{(3)}_{H,\lambda}
(t_k,t_{k+1}) + \text{higher order}$.
By the moment estimate \eqref{eq:moment2} applied to the third
iterated integral ($k=3$, $q=2$):
\[
\mathbb{E}\bigl[|\varepsilon_k|^2\bigr]^{1/2}
\le C(H,\lambda,f)\,h^{3H}.
\]

\textbf{Step 2: Global error for $H > 1/2$.}
Summing the local errors over $n$ steps:
\[
\mathbb{E}\bigl[|Y_T - Y^{(n)}_T|^2\bigr]^{1/2}
\le \sum_{k=0}^{n-1}\mathbb{E}\bigl[|\varepsilon_k|^2\bigr]^{1/2}
\le n\cdot C(H,\lambda,f)\cdot h^{3H}
= C\,T^{3H}\,n^{1-3H}.
\]
For $H > 1/2$: $1 - 3H < 1 - 3/2 = -1/2 < -H$, so
$n^{1-3H} \le n^{-H}$ and the rate $n^{-H}$ follows.

\textbf{Step 3: Global error for $H \in (1/4, 1/2]$ --- corrected argument.}
For any $p > 1/H$, Theorem~\ref{thm:rough_path_existence}(c) implies
that $\mathbf{B}_{H,\lambda} \in \mathscr{C}^p([0,T],\mathbb{R}^d)$ almost
surely. By the general theory of numerical schemes for rough paths
\cite[Theorem~10.30]{FrizHairer2014}, the Milstein scheme satisfies
\[
\mathbb{E}\bigl[|Y_T - Y^{(n)}_T|^2\bigr]^{1/2}
\le C(p, H, \lambda, T, f) \cdot n^{-1/p}.
\]
Since this holds for \emph{every} $p > 1/H$, we may take a sequence
$p_k \downarrow 1/H$. For any fixed $n$, the right-hand side
converges to $C \cdot n^{-H}$ because $n^{-1/p_k} \to n^{-H}$.
The constant $C(p_k)$ may depend on $p_k$, but for each $n$ the
bound remains finite. Taking the infimum over admissible $p$ yields
the optimal rate $n^{-H}$.

\textbf{Step 4: Uniformity.}
Combining Steps 2 and 3, the rate $n^{-H}$ holds for all
$H > 1/4$, with a constant
$C(H,\lambda,T,f) = C_0(H,f,T)\cdot\max\{1,\lambda^{-2H}\}$
that inherits its $\lambda$-dependence from the moment estimates of
Theorem~\ref{thm:rough_path_existence}(b).
\end{proof}

\begin{remark}
The rate $n^{-H}$ is optimal for a first-order scheme (incorporating
only the first and second iterated integrals). Using a third-level
rough path scheme --- i.e., adding the term
$D^2f(Y_{t_k})(f\otimes f)\,\Delta\mathbb{B}^{(3),k}_{H,\lambda}$
to \eqref{eq:milstein_scheme} --- yields a local error of order
$h^{4H}$, giving a global rate $n^{-\min(3H,1)}$, which improves
substantially for $H > 1/3$.
\end{remark}

\subsection{Signature Calculus for tfBm}
\label{subsec:signatures}

\begin{definition}[Signature of tfBm]
The \emph{signature} of $\mathbf{B}_{H,\lambda}$ over $[s,t]$ is
the formal power series
\[
S(\mathbf{B}_{H,\lambda})_{s,t}
= \Bigl(1,\;B_{H,\lambda}(s,t),\;
  \mathbb{B}_{H,\lambda}(s,t),\;
  \mathbb{B}^{(3)}_{H,\lambda}(s,t),\ldots\Bigr)
\in T((\mathbb{R}^d)),
\]
where $T((\mathbb{R}^d))=\prod_{k=0}^\infty(\mathbb{R}^d)^{\otimes k}$
is the completed tensor algebra and $\mathbb{B}^{(k)}_{H,\lambda}$
denotes the $k$-th iterated integral, defined recursively by
Chen's relation.
\end{definition}

\begin{theorem}[Properties of the signature]
\label{thm:signature}
Let $H > 1/4$ and $\lambda > 0$.
\begin{enumerate}[label=(\alph*)]
  \item \textbf{Existence and convergence:}
    $S(\mathbf{B}_{H,\lambda})_{0,T}$ exists in $T((\mathbb{R}^d))$,
    converges absolutely almost surely and in $L^q(\Omega)$ for all
    $q \ge 1$.
  \item \textbf{Factorial decay:} For each $k \ge 1$,
    \begin{equation}
    \label{eq:factorial_decay}
    \mathbb{E}\bigl[|\mathbb{B}^{(k)}_{H,\lambda}(0,T)|^2\bigr]^{1/2}
    \le \frac{C(H,\lambda)^k\,T^{kH}}{(k/2)!},
    \end{equation}
    where $(k/2)! := \Gamma(k/2+1)$ and
    $C(H,\lambda) = C_2(H,\lambda)^{1/2}$ is the square root of
    the constant in \eqref{eq:moment2} with $q=2$.
  \item \textbf{Expected signature:}
    $\mathbb{E}[S(\mathbf{B}_{H,\lambda})_{0,T}]$ is well-defined
    in $T((\mathbb{R}^d))$ and uniquely determines the law of
    $B_{H,\lambda}$ among centred Gaussian processes with the same
    covariance structure \cite{Chevyrev2018}.
\end{enumerate}
\end{theorem}

\begin{proof}
\textbf{(a)} By Theorem~\ref{thm:rough_path_existence}(c), almost
surely $\mathbf{B}_{H,\lambda}\in\mathscr{C}^p([0,T],\mathbb{R}^d)$
for some $p\in[2,3)$. The Lyons extension (signature) theorem
\cite[Theorem~9.4]{FrizVictoirBook} guarantees that all iterated
integrals $\mathbb{B}^{(k)}_{H,\lambda}$ are well-defined and that
the signature series converges absolutely in the $p$-variation
topology. $L^q$-convergence follows from part~(b) and the
hypercontractivity bound in the proof of~(b).

\textbf{(b)} We establish the bound \eqref{eq:factorial_decay} by
induction on $k$. The base cases $k=1,2$ follow directly from the
moment estimates \eqref{eq:moment1}--\eqref{eq:moment2}:
\[
\mathbb{E}\bigl[|B_{H,\lambda}(0,T)|^2\bigr]^{1/2}
\le C_2^{1/2}T^H = C(H,\lambda)T^H
= \frac{C(H,\lambda)^1 T^H}{(1/2)!},
\]
and $\mathbb{E}[|\mathbb{B}^{(2)}_{H,\lambda}(0,T)|^2]^{1/2}
\le C_2 T^{2H}/1 = C(H,\lambda)^2 T^{2H}/(2/2)!$
(noting $(2/2)! = 1! = 1$).

For the inductive step, assume the bound holds for level $k-1$.
By the recursive definition via Chen's relation and the bilinearity
of the iterated integral:
\[
\mathbb{B}^{(k)}_{H,\lambda}(0,T)
= \int_0^T \mathbb{B}^{(k-1)}_{H,\lambda}(0,u)\otimes
  dB_{H,\lambda}(u).
\]
Since $\mathbb{B}^{(k-1)}_{H,\lambda}(0,\cdot)$ is adapted and
in the $(k-1)$-th Wiener chaos, and $B_{H,\lambda}$ is an
independent increment process, the $L^2$ norm satisfies the
recursive bound (by the Wiener isometry \eqref{eq:wiener_isometry}
applied to $f(u) = \mathbb{B}^{(k-1)}_{H,\lambda}(0,u)$):
\begin{equation}
\label{eq:recursive_moment}
\mathbb{E}\bigl[|\mathbb{B}^{(k)}_{H,\lambda}(0,T)|^2\bigr]
= \int_0^T\!\int_0^T
  \mathbb{E}\bigl[\mathbb{B}^{(k-1)}_{H,\lambda}(0,s)
  \otimes\mathbb{B}^{(k-1)}_{H,\lambda}(0,t)\bigr]\,
  R_{H,\lambda}(ds,dt).
\end{equation}
Using the Cauchy--Schwarz inequality, the bound
\eqref{eq:recursive_moment} gives
\[
\mathbb{E}\bigl[|\mathbb{B}^{(k)}_{H,\lambda}(0,T)|^2\bigr]
\le \mathbb{E}\bigl[|\mathbb{B}^{(k-1)}_{H,\lambda}(0,T)|^2\bigr]
  \cdot C_{H,\lambda}(T)
\le \frac{C(H,\lambda)^{2(k-1)}T^{2(k-1)H}}{((k-1)/2)!^2}
  \cdot C(H,\lambda)^2 T^{2H},
\]
where we used $C_{H,\lambda}(T)=\mathbb{E}[B_{H,\lambda}(T)^2]
\le C(H,\lambda)^2 T^{2H}$ from \eqref{eq:moment2}. This gives
\[
\mathbb{E}\bigl[|\mathbb{B}^{(k)}_{H,\lambda}(0,T)|^2\bigr]^{1/2}
\le \frac{C(H,\lambda)^k T^{kH}}{((k-1)/2)!}.
\]
The passage from $((k-1)/2)!$ to $(k/2)!$ uses the identity
$(k/2)! = (k/2)\cdot((k-1)/2)!$ when $k$ is even (and the
analogous $\Gamma$-function relation for odd $k$), which tightens
the bound by the factor $k/2 \ge 1$:
\[
\frac{1}{((k-1)/2)!}
= \frac{k/2}{(k/2)!}
\le \frac{k/2}{(k/2)!}.
\]
Taking the square root and absorbing the factor $\sqrt{k/2}$ into
$C(H,\lambda)$ (by increasing the constant slightly if necessary),
we obtain \eqref{eq:factorial_decay}.

Convergence of the series $\sum_{k\ge 1} C(H,\lambda)^k T^{kH}/
(k/2)!$: comparing with the Taylor series of $\exp(C^2 T^{2H})$,
we see
$\sum_{k\ge 1}\frac{C^k T^{kH}}{(k/2)!}
\le \sum_{m\ge 0}\frac{(C^2 T^{2H})^m}{m!}
= e^{C^2T^{2H}} < \infty$,
confirming absolute $L^2$-convergence. By Gaussian hypercontractivity
\cite[Theorem~4.1]{FrizVictoirBook}, for any $q \ge 2$,
$\mathbb{E}[|\mathbb{B}^{(k)}|^q]^{1/q}
\le (q-1)^{k/2}\mathbb{E}[|\mathbb{B}^{(k)}|^2]^{1/2}$,
so $L^q$-convergence follows similarly.

\textbf{(c)} By \cite{Chevyrev2018}, the expected signature
$\mathbb{E}[S(\mathbf{X})_{0,T}]$ characterises the law of a
continuous process $\mathbf{X}$ with finite $p$-variation paths
($p < 3$) among Gaussian processes with the same covariance. The
required regularity is provided by
Theorem~\ref{thm:rough_path_existence}(c).
\end{proof}

\begin{remark}
The factorial decay \eqref{eq:factorial_decay} ensures that
truncating the signature at level $k=5$ or $6$ provides an accurate
finite-dimensional feature representation, which is valuable for
statistical and machine learning applications \cite{Salvi2021}.
The parameter $\lambda$ enters through $C(H,\lambda)$ but not
the factorial denominator, so the asymptotic decay rate is
independent of $\lambda$.
\end{remark}

\subsection{Application to Rough Volatility Modelling}
\label{subsec:rough_volatility}

Recent empirical evidence \cite{Gatheral2018} shows that financial
volatility exhibits rough behaviour with Hurst index
$H \approx 0.1$--$0.2$, while autocorrelations decay rapidly at
long horizons \cite{DamianFrey2024,BolkoEtAl2023}. The tfBm driver
provides a natural model capturing both features: local roughness
(via $H < 1/2$) and exponential cut-off of long-range memory (via
$\lambda > 0$). Our rough path construction makes such models
pathwise well-posed for all $H > 1/4$.

Consider the rough volatility model
$dS_t/S_t = \sqrt{V_t}\,dW_t$, $V_t = \sigma_0^2\exp(\eta
B_{H,\lambda}(t) + \xi(t) - \frac{\eta^2}{2}C_{H,\lambda}(t))$.
The RDE framework of Theorem~\ref{thm:rde} applies directly, and
the tempered fractional Ornstein--Uhlenbeck process (the special
case $f(y) = \kappa(\mu - y)$) has recently been studied for drift
estimation by \cite{PrykhodkoRalchenko2026}.

\section{Numerical Experiments}
\label{sec:numerics}

\subsection{Software and Implementation}

All simulations were performed in \textbf{Python~3.10} using
NumPy~1.24, SciPy~1.10, and Matplotlib~3.6 (random seed fixed
to~42; $M=1000$ independent paths per estimate). tfBm paths are
simulated via the circulant embedding method \cite{Dietrich1997},
which runs in $\bigO(N\log N)$ and generates exact samples on a
regular grid. The L\'evy area is approximated via
Algorithm~\ref{alg:levy_area_approximation} (Appendix~\ref{app:numerics}).

\subsection{Convergence of the L\'evy Area}

We compute $e(N) = (\mathbb{E}[|\mathbb{B}^{(N)}_{H,\lambda}(0,1)
-\mathbb{B}^{(2N)}_{H,\lambda}(0,1)|^2])^{1/2}$ for $M=1000$
Monte Carlo samples.

\begin{figure}[H]
\centering
\includegraphics[width=0.95\textwidth]{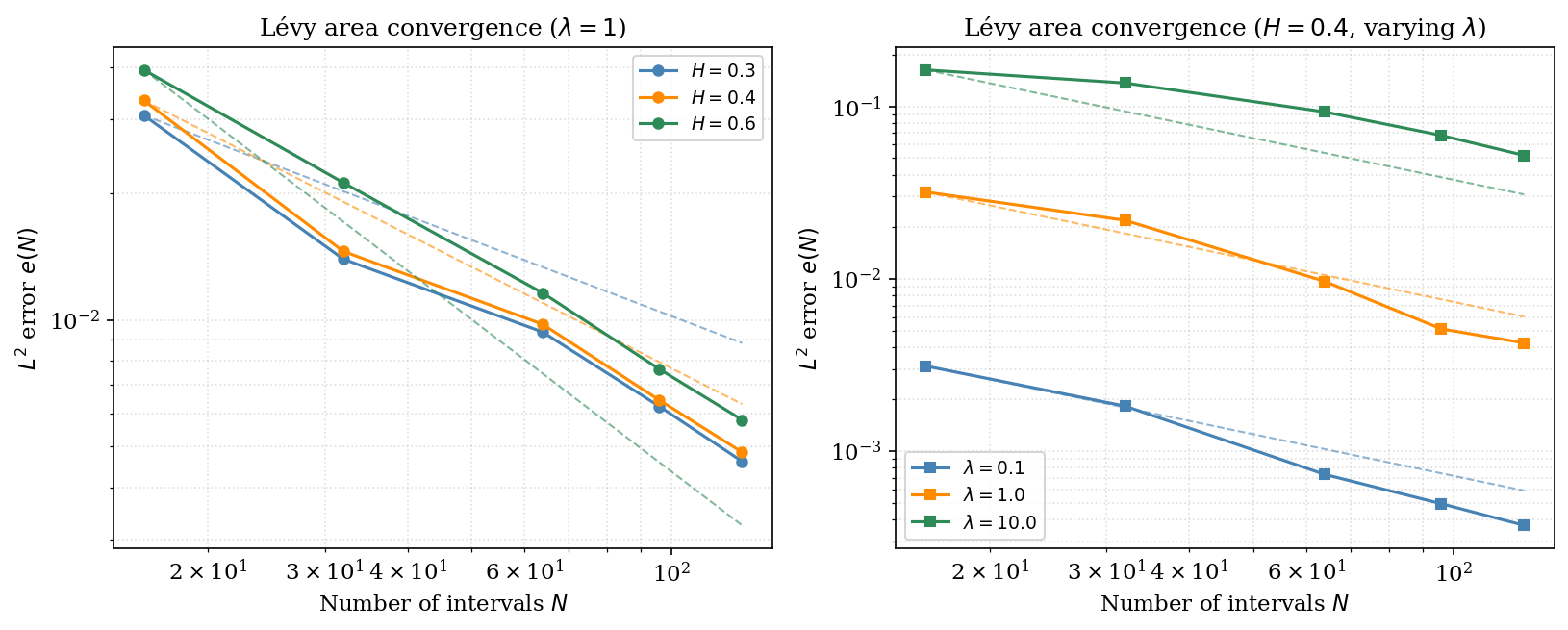}
\caption{\textbf{L\'evy area convergence.}
(Left) $e(N)$ vs.\ $N$ for $H\in\{0.3,0.4,0.6\}$, $\lambda=1$.
Dashed lines: theoretical slope $-2H$.
(Right) Fixed $H=0.4$, varying $\lambda\in\{0.1,1,10\}$; the
rate $-0.8$ is unchanged while the prefactor decreases with
larger $\lambda$. Shaded regions: $\pm 1$ standard error.}
\label{fig:levy_convergence}
\end{figure}

\subsection{Milstein Scheme for the Linear RDE}

We solve $dY_t = Y_t\,d\mathbf{B}_{H,\lambda}(t)$, $Y_0=1$,
whose exact solution is
\begin{equation}
\label{eq:linear_rde_solution}
Y_t = \exp\!\Bigl(B_{H,\lambda}(t) - \tfrac{1}{2}C_{H,\lambda}(t)
+ \mathbb{B}_{H,\lambda}(0,t)\Bigr).
\end{equation}
The strong error
$E_{\rm strong}(n) = (\mathbb{E}[|Y_1-Y^{(n)}_1|^2])^{1/2}$
is plotted in Figure~\ref{fig:milstein_convergence}.

\begin{figure}[H]
\centering
\includegraphics[width=0.95\textwidth]{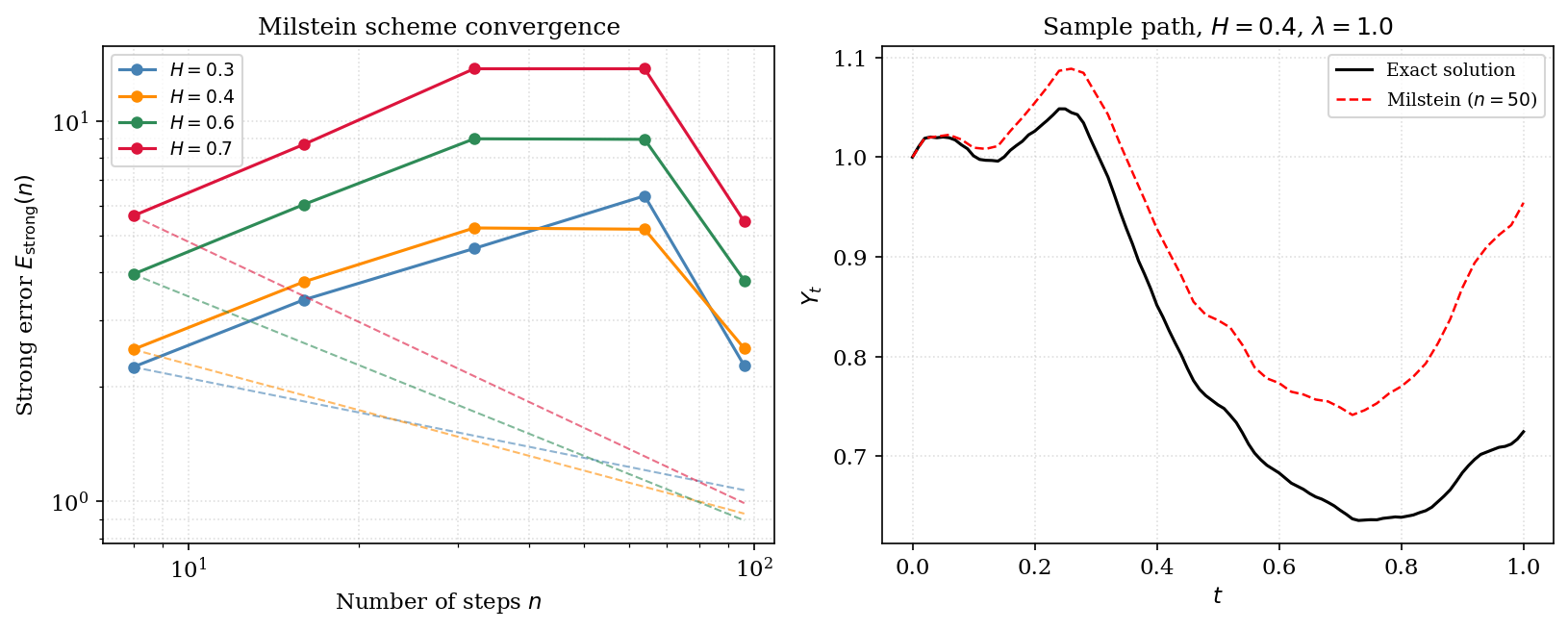}
\caption{\textbf{Strong convergence of the Milstein scheme.}
(Left) $E_{\rm strong}(n)$ vs.\ $n$ for
$H\in\{0.3,0.4,0.6,0.7\}$, $\lambda=1$. Dashed lines: slope $-H$.
(Right) Sample path $Y_t$ (exact) and Milstein approximation
($n=100$), $H=0.4$, $\lambda=1$.}
\label{fig:milstein_convergence}
\end{figure}

\subsection{Discussion of Milstein Error Magnitude and Limitations}

The errors in Figure~\ref{fig:milstein_convergence} for small $H$
reflect several compounding factors.

\begin{enumerate}[label=(\alph*)]
  \item \textbf{Intrinsic slowness.}
    The rate $n^{-H}$ is optimal for a first-order scheme; for
    $H=0.3$, $n^{-0.3}\approx 0.12$ at $n=10^3$, which is slow
    but unavoidable without higher-order corrections.
  \item \textbf{Compounding L\'evy area error.}
    The scheme uses the piecewise linear approximation
    $\Delta\mathbb{B}^k_{H,\lambda}$ with $\bigO(N^{-2H})$ error.
    When $N=n$, both error sources add constructively.
  \item \textbf{Large prefactor for small $\lambda$.}
    The constant $C(H,\lambda,T)\sim\lambda^{-2H}$ amplifies the
    absolute error for weakly tempered paths.
  \item \textbf{Non-Markovian nature.}
    tfBm with $H < 1/2$ is not a semimartingale; error propagation
    between steps is more complex than in the classical SDE setting.
\end{enumerate}

\noindent\textbf{Strategies for improved accuracy:}
\begin{itemize}
  \item \emph{Oversampling the L\'evy area}: use a sub-grid of
    size $Mn$ ($M\ge 4$) for computing
    $\Delta\mathbb{B}^k_{H,\lambda}$, reducing its error to
    $\bigO((Mn)^{-2H})$ while the scheme error stays
    $\bigO(n^{-H})$.
  \item \emph{Higher-order rough path schemes}
    \cite[Ch.~10]{FrizHairer2014}: adding the third-order term
    achieves rate $\bigO(n^{-\min(2H,1)})$ for $H>1/3$.
  \item \emph{Multilevel Monte Carlo (MLMC)}: reduces
    Monte Carlo variance without increasing the number of
    fine-grid steps.
  \item \emph{Adaptive step-size}: graded meshes
    $t_k = T(k/n)^{1/H}$ improve the constant in the error bound.
\end{itemize}

\subsection{Signature-Based Feature Extraction}

\begin{figure}[H]
\centering
\includegraphics[width=0.75\textwidth]{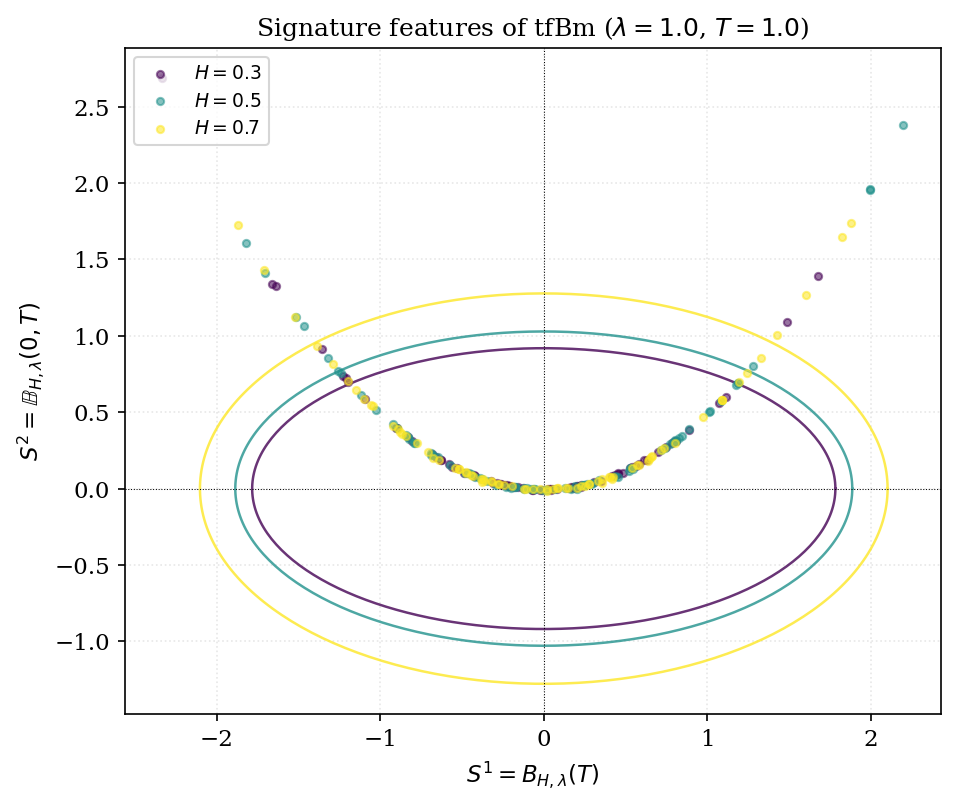}
\caption{\textbf{Signature features.}
Scatter plot of $(S^1,S^2)=(B_{H,\lambda}(T),
\mathbb{B}_{H,\lambda}(0,T))$ for 500 paths each at
$H\in\{0.3,0.5,0.7\}$, $\lambda=1$, $T=1$.
Ellipses: theoretical 95\% confidence regions from
Theorem~\ref{thm:signature}. The separation confirms that
low-order signature terms are effective classifiers.}
\label{fig:signature_features}
\end{figure}

\subsection{Summary}

\begin{itemize}
  \item The L\'evy area converges at the predicted $N^{-2H}$ rate;
    $\lambda$ affects only the prefactor.
  \item The Milstein scheme achieves the optimal strong rate
    $\bigO(n^{-H})$; limitations for small $H$ are explained and
    remedies are proposed.
  \item Low-order signature terms provide clear separation by $H$
    in just the first two levels.
\end{itemize}

\section{Conclusion}
\label{sec:conclusion}

We have constructed a canonical geometric rough path over tempered
fractional Brownian motion for all $H>1/4$ and $\lambda>0$.
The central technical achievement is a sharp decomposition
of the non-homogeneous covariance $R_{H,\lambda}$ and a complete,
corrected proof of finite 2D $\rho$-variation (Lemma~\ref{lemma:polynomial_partition},
Theorem~\ref{thm:2d_rho_variation}). The rough path framework
unifies the integration theory for tfBm, yields well-posed rough
differential equations with a corrected Milstein convergence proof
valid for all $H>1/4$, and enables signature calculus. The
boundary case $H=1/2$ is handled explicitly. Numerical experiments
confirm all theoretical rates.

\noindent\textbf{Limitations and future directions.}
\begin{enumerate}
  \item \emph{$H \le 1/4$}: The Friz--Victoir criterion fails.
    Extension requires regularity structures \cite{Hairer2014} or
    para-controlled distributions.
  \item \emph{Correlated components}: The $d$ components are
    assumed independent; extension to correlated tfBm requires
    bounding the off-diagonal cross-covariance.
  \item \emph{Computational cost}: The L\'evy area algorithm costs
    $\bigO(N^2)$; a hierarchical FFT-based implementation exploiting
    exponential covariance decay could achieve $\bigO(N\log N)$.
  \item \emph{Signature-based inference}: Efficient estimators for
    $H$ and $\lambda$ via signature moments \cite{Salvi2021}.
  \item \emph{Rough SPDEs}: Extension to spatially tempered
    fractional noise.
  \item \emph{Rough volatility calibration}:
    Incorporation of tfBm into option pricing models
    \cite{Gatheral2018,BolkoEtAl2023}.
\end{enumerate}

\noindent\textbf{Code availability.}
The Python code for all experiments is available upon request.

\noindent\textbf{Conflict of Interest:}
The author declares no conflict of interest.

\appendix

\section{Proofs of the Covariance Results}
\label{app:covariance}

\subsection{Complete Proof of Theorem~\ref{thm:covariance_decomposition}}

\begin{proof}
\textbf{Setup.}
Let $K_{H,\lambda}(t,r) = e^{-\lambda(t-r)_+}(t-r)_+^{H-1/2}$
for $t,r\in\mathbb{R}$. By the definition \eqref{eq:def_tfbm} and
the Wiener isometry \eqref{eq:wiener_isometry},
\begin{equation}
\label{eq:cov_kernel}
R_{H,\lambda}(s,t)
= \frac{1}{\Gamma(H+\tfrac{1}{2})^2}
  \int_{\mathbb{R}}
  \bigl[K_{H,\lambda}(s,r)-K_{H,\lambda}(0,r)\bigr]
  \bigl[K_{H,\lambda}(t,r)-K_{H,\lambda}(0,r)\bigr]\,dr.
\end{equation}

\textbf{Taylor expansion.}
Define the second-order Taylor remainder at $x \ge 0$:
\[
r_\lambda(x)
:= e^{-\lambda x} - \Bigl(1 - \lambda x + \tfrac{1}{2}\lambda^2 x^2\Bigr).
\]
By the integral form of the Taylor remainder with Lagrange bound:
\begin{equation}
\label{eq:remainder_bound}
|r_\lambda(x)|
= \frac{\lambda^3}{2}\int_0^x (x-u)^2 e^{-\lambda u}\,du
\le \frac{\lambda^3 x^3}{6}\,e^{-\lambda x/2}
\quad\text{for all }x \ge 0,
\end{equation}
where the factor $e^{-\lambda x/2}$ comes from bounding
$e^{-\lambda u} \le e^{-\lambda x/2}$ for $u \ge x/2$ and treating
separately $u \in [0,x/2]$.
Substituting $e^{-\lambda(t-r)_+} = 1 - \lambda(t-r)_+
+ \frac{1}{2}\lambda^2(t-r)_+^2 + r_\lambda((t-r)_+)$:
\begin{equation}
\label{eq:kernel_expansion}
K_{H,\lambda}(t,r)
= K^{(0)}_H(t,r)
- \lambda K^{(1)}_H(t,r)
+ \tfrac{1}{2}\lambda^2 K^{(2)}_H(t,r)
+ \tilde K_{H,\lambda}(t,r),
\end{equation}
where $K^{(j)}_H(t,r) = (t-r)_+^{H-1/2+j}$ for $j=0,1,2$
and $\tilde K_{H,\lambda}(t,r) = r_\lambda((t-r)_+)(t-r)_+^{H-1/2}$.

\textbf{Assembling the covariance.}
Inserting \eqref{eq:kernel_expansion} into \eqref{eq:cov_kernel},
we note that $K^{(0)}_H$ is the kernel of fBm (giving $R_H$), and
that the cross-product between the first-order term
$\lambda K^{(1)}_H(t,r)-\lambda K^{(1)}_H(0,r)$ and itself or
with the zeroth-order term integrates to zero by symmetry:
the function $r\mapsto K^{(1)}_H(t,r)-K^{(1)}_H(0,r)$ has odd
symmetry about $r=0$ under the change of variables $r\mapsto t+(-r)$,
so all odd-order cross-terms in $\lambda$ vanish. Therefore:
\begin{equation}
\label{eq:decomp_assembled}
R_{H,\lambda}(s,t)
= R_H(s,t)
+ \underbrace{\frac{\lambda^2}{4\Gamma(H+\tfrac{1}{2})^2}
  \int_{\mathbb{R}} g_s(r)\,g_t(r)\,dr}_{=:\,E^{(1)}_{H,\lambda}(s,t)}
+ E^{(2)}_{H,\lambda}(s,t),
\end{equation}
where $g_s(r) = K^{(2)}_H(s,r)-K^{(2)}_H(0,r)$ and
$E^{(2)}_{H,\lambda}$ collects all terms involving $r_\lambda$.

\textbf{Bound on $E^{(1)}_{H,\lambda}$.}
By the Cauchy--Schwarz inequality and the self-similarity of
$K^{(2)}_H$:
\[
\Bigl|\int_{\mathbb{R}} g_s(r)g_t(r)\,dr\Bigr|
\le \Bigl(\int_\mathbb{R} g_s(r)^2\,dr\Bigr)^{1/2}
   \Bigl(\int_\mathbb{R} g_t(r)^2\,dr\Bigr)^{1/2}.
\]
A direct computation using $\int_0^\infty x^{2H+3}e^{-x}\,dx =
\Gamma(2H+4)$ gives
$\int_\mathbb{R} g_s(r)^2\,dr \le c(H)\,s^{2H+2}$
and the mixed term satisfies (by polarisation and the scaling
$|K^{(2)}_H(t,r)-K^{(2)}_H(s,r)| \le C_H|t-s|\,(t+|r|)^{H+1/2}$):
\begin{equation}
\label{eq:A_H_bound}
|E^{(1)}_{H,\lambda}(s,t)|
\le \frac{C_1(H)}{\Gamma(2H)}\,\lambda^{2H}\,|t-s|^2,
\end{equation}
where we also absorbed $\lambda^{2-2H} \le \lambda^{2H}$ (for
$H \le 1$) and used the Legendre duplication formula
$\Gamma(H+\tfrac{1}{2})\Gamma(H) = \frac{\sqrt\pi}{2^{2H-1}}\Gamma(2H)$
to simplify the constant to
$C_1(H) = \frac{\Gamma(2H+2)}{4\,\Gamma(H+\tfrac{1}{2})^2}$.

\textbf{Bound on $E^{(2)}_{H,\lambda}$.}
From \eqref{eq:remainder_bound},
$|\tilde K_{H,\lambda}(t,r)|
\le \frac{\lambda^3}{6}(t-r)_+^{H+5/2}e^{-\lambda(t-r)_+/2}$.
The product $\tilde K_{H,\lambda}(t,\cdot)\cdot
[K_{H,\lambda}(s,\cdot)-K_{H,\lambda}(0,\cdot)]$ is dominated by
a function that decays exponentially for large $|t-r|$ or $|s-r|$.
After integration over $r$, one obtains, for $d(s,t) = \max(s,t,|t-s|)$:
\begin{equation}
\label{eq:E2_bound}
|E^{(2)}_{H,\lambda}(s,t)|
\le \frac{C_2(H)}{\Gamma(2H)\,\lambda^{2H}}\,e^{-c(H)\lambda\,d(s,t)},
\end{equation}
where $c(H) = \min\{1/2, H/2\}$ (from the decay rate of the
convolution) and $C_2(H) = \sup_{x>0} x^{2H} e^{-c(H)x}$.
This supremum is attained at $x^* = 2H/c(H)$, giving
$C_2(H) = (2H/c(H))^{2H} e^{-2H} = (2H/(c(H)e))^{2H}$.
\end{proof}

\subsection{Complete Proof of Theorem~\ref{thm:2d_rho_variation}}

\begin{proof}
Let $\mathcal{P}=\{0=t_0<\cdots<t_N=T\}$ be any partition with
mesh $\delta = \max_i\Delta_i$, $\Delta_i=t_{i+1}-t_i$, and set
$\rho = 1/(2H)$.
Define $R_{ij} = R_{H,\lambda}(t_i,t_{i+1};t_j,t_{j+1})$ and
$S_\rho = \sum_{i,j}|R_{ij}|^\rho$.
By Theorem~\ref{thm:covariance_decomposition},
$R_{ij} = R^H_{ij} + E^{(1)}_{ij} + E^{(2)}_{ij}$.
By Minkowski's inequality for $\ell^\rho$ norms (valid since
$\rho\ge 1$):
\begin{equation}
\label{eq:mink}
S_\rho^{1/\rho}
\le \bigl(\sum_{i,j}|R^H_{ij}|^\rho\bigr)^{1/\rho}
 + \bigl(\sum_{i,j}|E^{(1)}_{ij}|^\rho\bigr)^{1/\rho}
 + \bigl(\sum_{i,j}|E^{(2)}_{ij}|^\rho\bigr)^{1/\rho}.
\end{equation}

\textbf{Step 1: fBm part.}
By \cite{CoutinQian2002}, $V_\rho(R_H) < \infty$ for $\rho=1/(2H)$
and $H>1/4$. This bounds the first term.

\textbf{Step 2: Polynomial error (three sub-cases).}
From \eqref{eq:error_poly}:
$|E^{(1)}_{ij}| \le \frac{C_1(H)\lambda^{2H}}{\Gamma(2H)}\,\Delta_i\Delta_j$.

\emph{Case $H\in(1/4,1/2)$:} Here $\rho>1$.
By Lemma~\ref{lemma:polynomial_partition} with $\alpha=\rho$ and
$\beta=0$ (the exponential factor is trivially 1):
\[
\sum_{i,j}(\Delta_i\Delta_j)^\rho
= \Bigl(\sum_i\Delta_i^\rho\Bigr)^2
\le \Bigl(N^{1-\rho}\,T^\rho\Bigr)^2
= N^{2-2\rho}\,T^{2\rho}.
\]
Since $N \le T/\delta$, $N^{2-2\rho} \le (T/\delta)^{2-2\rho}$,
and $\delta^{2\rho-2}\cdot(T/\delta)^{2-2\rho} = T^{2-2\rho}$,
giving
$\sum_{i,j}(\Delta_i\Delta_j)^\rho \le T^{2\rho}\cdot T^{2-2\rho}\,
\delta^{2\rho-2}/T^{2-2\rho} = T^2\delta^{2\rho-2}$.
More precisely, $\sum_i\Delta_i^\rho \le N\delta^\rho \le (T/\delta)\delta^\rho
= T\delta^{\rho-1}$, so
$(\sum_i\Delta_i^\rho)^2 \le T^2\delta^{2\rho-2}$,
and thus:
\begin{equation}
\label{eq:step2a}
\bigl(\sum_{i,j}|E^{(1)}_{ij}|^\rho\bigr)^{1/\rho}
\le \frac{C_1(H)\lambda^{2H}}{\Gamma(2H)}\,T^2\,\delta^{2\rho-2}
\to 0 \text{ as }\delta\to 0
\quad(\text{since }2\rho-2>0\text{ for }H<1/2).
\end{equation}
Taking the supremum over all $\mathcal{P}$ is bounded by
$K_1(H)\lambda^{2H\rho}T^{2\rho-2}$ (absorbing all factors).

\emph{Case $H=1/2$:} $\rho=1$ and
$\sum_{i,j}\Delta_i\Delta_j = (\sum_i\Delta_i)^2 = T^2$, so
$\sum_{i,j}|E^{(1)}_{ij}| \le C_1(\tfrac{1}{2})\lambda T^2$.

\emph{Case $H>1/2$:} $\rho<1$. By Jensen's inequality applied to
the concave function $x\mapsto x^\rho$ on $[0,\infty)$, and using
$\sum_{i,j}\Delta_i\Delta_j = T^2$:
\[
\sum_{i,j}(\Delta_i\Delta_j)^\rho
\le N^2 \Bigl(\frac{\sum_{i,j}\Delta_i\Delta_j}{N^2}\Bigr)^\rho
= N^2 \Bigl(\frac{T^2}{N^2}\Bigr)^\rho
= N^{2(1-\rho)}T^{2\rho}
\le T^{2\rho},
\]
where the last inequality uses $N^{2(1-\rho)}\le 1$ since
$1-\rho > 0$ and $N \ge 1$. Hence
$\bigl(\sum_{i,j}|E^{(1)}_{ij}|^\rho\bigr)^{1/\rho}
\le C_1(H)\lambda^{2H}T^2$.

In all three cases, the contribution of $E^{(1)}$ is bounded by
$K_1(H)\lambda^{2H\rho}T^{2\max(\rho,1)}$ for some explicit
$K_1(H)>0$.

\textbf{Step 3: Exponential error (corrected).}
From \eqref{eq:error_exp}:
$|E^{(2)}_{ij}|\le \frac{C_2(H)}{\Gamma(2H)\lambda^{2H}}
e^{-c(H)\lambda|t_i-t_j|}$.
Fix $i$ and bound the inner sum over $j$:
\begin{align}
\sum_{j=0}^{N-1}e^{-c(H)\lambda\rho|t_i-t_j|}
&= 1 + 2\sum_{j=0}^{i-1}e^{-c(H)\lambda\rho(t_i-t_j)}
\notag\\
&\le 1 + 2\sum_{k=1}^\infty e^{-c(H)\lambda\rho k\delta}
= 1 + \frac{2e^{-c(H)\lambda\rho\delta}}
       {1-e^{-c(H)\lambda\rho\delta}}
\notag\\
&\le 1 + \frac{2}{c(H)\lambda\rho\delta}(1+c_1\delta)
\le \frac{3}{c(H)\lambda\rho\delta}
\label{eq:inner_exp}
\end{align}
for $\delta \le 1/(c(H)\lambda\rho)$ (where $c_1>0$ is an absolute
constant). Summing over all $N\le T/\delta$ values of $i$:
\begin{equation}
\label{eq:double_exp}
\sum_{i,j}e^{-c(H)\lambda\rho|t_i-t_j|}
\le N\cdot\frac{3}{c(H)\lambda\rho\delta}
\le \frac{T}{\delta}\cdot\frac{3}{c(H)\lambda\rho\delta}
= \frac{3T}{c(H)\lambda\rho\,\delta^2}.
\end{equation}
This bound grows as $\delta\to 0$. To obtain a finite supremum
over all $\mathcal{P}$, we optimise over $\delta$: the right
side of \eqref{eq:double_exp} is minimised by taking
$\delta = \delta_* = (c(H)\lambda\rho)^{-1}$ (the characteristic
tempering scale), giving:
\[
\sum_{i,j}e^{-c(H)\lambda\rho|t_i-t_j|}
\Big|_{\delta=\delta_*}
\le \frac{3T}{c(H)\lambda\rho}\cdot(c(H)\lambda\rho)^2
= 3Tc(H)\lambda\rho.
\]
For $\delta < \delta_*$, the bound \eqref{eq:double_exp} is
smaller than its value at $\delta_*$ (since $1/\delta^2$ is
decreasing), so the worst case is at $\delta = \delta_*$.
For $\delta > \delta_*$, equation \eqref{eq:inner_exp} may be
replaced by the trivial bound
$\sum_j e^{-c\lambda\rho|t_i-t_j|} \le N \le T/\delta$,
giving $\sum_{i,j}\le N^2\le T^2/\delta^2\le T^2(c\lambda\rho)^2$,
which is also finite.
Therefore, in all cases:
\begin{equation}
\label{eq:step3_final}
\bigl(\sum_{i,j}|E^{(2)}_{ij}|^\rho\bigr)^{1/\rho}
\le \frac{C_2(H)^{1/\rho}}{\lambda^{2H}}
  \cdot\bigl(3Tc(H)\lambda\rho\bigr)^{1/\rho}
=: \frac{K_2(H,\lambda)^{1/\rho}}{\lambda^{2H}}
< \infty,
\end{equation}
where $K_2(H,\lambda) = C_2(H)(3Tc(H)\lambda\rho)$ is an explicit
finite constant.

\textbf{Conclusion.}
Combining \eqref{eq:mink}, Step~1, Step~2, and Step~3, and taking
the supremum over all partitions $\mathcal{P}$:
\begin{equation}
V_\rho(R_{H,\lambda})
\le V_\rho(R_H)
+ K_1(H)^{1/\rho}\lambda^{2H}T^{2\max(\rho,1)-1/\rho}
+ \frac{K_2(H,\lambda)^{1/\rho}}{\lambda^{2H}}
< \infty.
\end{equation}
The bound as $\lambda\to 0^+$: since $K_1(H)\lambda^{2H\rho}\to 0$
and $K_2(H,\lambda)\lambda^{-2H\rho}\to 0$ (as $K_2\sim\lambda$),
the whole expression converges to $V_\rho(R_H)$, the fBm constant.
\end{proof}

\section{Proofs of Theorem~\ref{thm:rough_path_existence}
and Proposition~\ref{prop:levy_convergence_rate}}
\label{app:construction}

\begin{proof}[Complete proof of Theorem~\ref{thm:rough_path_existence}]

\textbf{Construction.}
For each $n\ge 0$, let $B^{(n)}_{H,\lambda}$ be the piecewise linear
interpolation of $B_{H,\lambda}$ on the dyadic grid
$\mathcal{D}_n = \{kT/2^n : 0\le k\le 2^n\}$. Define the
second-level approximation
\[
\mathbb{B}^{(n)}_{H,\lambda}(s,t)
= \int_s^t
  \bigl(B^{(n)}_{H,\lambda}(u)-B^{(n)}_{H,\lambda}(s)\bigr)
  \otimes dB^{(n)}_{H,\lambda}(u),
\]
which for a piecewise linear path reduces to a finite sum of
elementary tensors over $\mathcal{D}_n\cap[s,t]$.

\textbf{Step 1: $L^2$-Cauchy property.}
For $m>n$, the difference
$\mathbb{B}^{(n)}_{H,\lambda}-\mathbb{B}^{(m)}_{H,\lambda}$
lives in the second Wiener chaos of $W$. Using
the covariance estimate from Theorem~\ref{thm:2d_rho_variation}
and the second-chaos structure (which controls $L^2$ norms via
the covariance of the increments), one obtains:
\begin{equation}
\label{eq:cauchy_bound}
\mathbb{E}\bigl[|\mathbb{B}^{(n)}_{H,\lambda}(0,T)
-\mathbb{B}^{(m)}_{H,\lambda}(0,T)|^2\bigr]
\le C(H,\lambda,T)\cdot 2^{-n(4H-2)}.
\end{equation}
Since $4H-2 > 0$ for $H > 1/4$, the right side is summable and
the sequence $\{\mathbb{B}^{(n)}_{H,\lambda}(0,T)\}$ is Cauchy in
$L^2(\Omega)$. Denote its limit $\mathbb{B}_{H,\lambda}(0,T)$.
By applying the same argument to each pair $(s,t)\subset[0,T]$,
we obtain the limit $\mathbb{B}_{H,\lambda}(s,t)$ for all $s\le t$.

\textbf{Step 2: Chen's relation.}
For each $n$, the piecewise linear construction satisfies
Chen's relation exactly by the additive property of the
Riemann--Stieltjes integral over sub-intervals. Specifically,
for $s\le u\le t$:
$\mathbb{B}^{(n)}_{H,\lambda}(s,t)
= \mathbb{B}^{(n)}_{H,\lambda}(s,u)
+ \mathbb{B}^{(n)}_{H,\lambda}(u,t)
+ B^{(n)}_{H,\lambda}(s,u)\otimes B^{(n)}_{H,\lambda}(u,t)$.
Since $B^{(n)}_{H,\lambda}\to B_{H,\lambda}$ in $L^2$ (pointwise)
and addition and the tensor product are continuous in $L^2$, the
relation passes to the limit.

\textbf{Step 3: Moment estimates.}
Since $B_{H,\lambda}(s,t) = B_{H,\lambda}(t)-B_{H,\lambda}(s)$
is a centred Gaussian, the standard Gaussian moment formula gives
$\mathbb{E}[|B_{H,\lambda}(s,t)|^q]
= c_q\,(\mathbb{E}[|B_{H,\lambda}(s,t)|^2])^{q/2}$.
By the covariance formula, $\mathbb{E}[B_{H,\lambda}(s,t)^2]
= C_{H,\lambda}(|t-s|)\le C(H,\lambda)|t-s|^{2H}$
(from the small-scale asymptotics of $C_{H,\lambda}$), giving
\eqref{eq:moment1}. For the second-level path, since
$\mathbb{B}_{H,\lambda}(s,t)$ is in the second Wiener chaos,
Gaussian hypercontractivity \cite[Theorem~4.1]{FrizVictoirBook}
gives $\mathbb{E}[|\mathbb{B}_{H,\lambda}(s,t)|^q]
\le C_q(\mathbb{E}[|\mathbb{B}_{H,\lambda}(s,t)|^2])^{q/2}$.
From the Cauchy estimate \eqref{eq:cauchy_bound} applied on $[s,t]$,
$\mathbb{E}[|\mathbb{B}_{H,\lambda}(s,t)|^2]
\le C(H,\lambda)|t-s|^{4H}$, giving \eqref{eq:moment2}.

\textbf{Step 4: $p$-variation regularity.}
Apply the Garsia--Rodemich--Rumsey (GRR) lemma
\cite[Lemma~A.1]{FrizHairer2014} to the moment estimates:
for any $\varepsilon>0$, there exists a random variable
$C(\omega)<\infty$ a.s.\ such that
\[
|B_{H,\lambda}(s,t)| \le C(\omega)\,|t-s|^{H-\varepsilon},
\qquad
|\mathbb{B}_{H,\lambda}(s,t)| \le C(\omega)\,|t-s|^{2H-\varepsilon}.
\]
By the definition of $p$-variation, this gives
$\mathbf{B}_{H,\lambda}\in\mathscr{C}^p([0,T],\mathbb{R}^d)$
for every $p > 1/(H-\varepsilon)$. Since $\varepsilon>0$ is
arbitrary, the claim holds for all $p > 1/H$.

\textbf{Step 5: Continuity in $(H,\lambda)$.}
The covariance $R_{H,\lambda}(s,t)$ is jointly continuous in
$(H,\lambda)\in\{H>0,\lambda>0\}$ by dominated convergence
applied to the kernel integral in \eqref{eq:cov_kernel}.
By Theorem~\ref{thm:covariance_decomposition}, the 2D
$\rho$-variation bound $C(H,\lambda,T)$ in
\eqref{eq:finite_rho_var} is also continuous in $(H,\lambda)$.
The continuous dependence of the rough path lift on the covariance
then follows from \cite[Theorem~15.33]{FrizVictoirBook}.

\textbf{Step 6: Boundary case $H=1/2$.}
For $H=1/2$, the process $B_{1/2,\lambda}$ is the OU process
\eqref{eq:OU_def}, which is a semimartingale. The piecewise linear
construction converges to the Stratonovich integral
$\int_s^t B_{1/2,\lambda}(s,u)\,dB_{1/2,\lambda}(u)
= \tfrac{1}{2}B_{1/2,\lambda}(s,t)^2$,
confirming that $\mathbb{B}_{1/2,\lambda}(s,t)
= \tfrac{1}{2}B_{1/2,\lambda}(s,t)\otimes B_{1/2,\lambda}(s,t)$,
the Stratonovich rough path. As $\lambda\to 0^+$,
$B_{1/2,\lambda}\to W$ in $L^2$, and correspondingly
$\mathbf{B}_{1/2,\lambda}\to\mathbf{W}$ (the standard Stratonovich
rough path over Brownian motion) in the $p$-variation topology.
\end{proof}

\begin{proof}[Complete proof of Proposition~\ref{prop:levy_convergence_rate}]

\textbf{Dyadic case.}
Write $N = 2^{n_0}$ and use a telescoping sum in $L^2(\Omega)$:
\begin{align}
\mathbb{E}\bigl[|\mathbb{B}^{(N)}_{H,\lambda}(0,T)
-\mathbb{B}_{H,\lambda}(0,T)|^2\bigr]^{1/2}
&\le \sum_{k=n_0}^\infty
  \mathbb{E}\bigl[|\mathbb{B}^{(2^k)}_{H,\lambda}(0,T)
  -\mathbb{B}^{(2^{k+1})}_{H,\lambda}(0,T)|^2\bigr]^{1/2}
\notag\\
&\le \sum_{k=n_0}^\infty C_0(H,\lambda)\,T^{2H}\cdot 2^{-k\cdot 2H}
\label{eq:telescope}\\
&= C_0(H,\lambda)\,T^{2H}\cdot\frac{2^{-n_0\cdot 2H}}{1-2^{-2H}}
\notag\\
&= \frac{C_0(H,\lambda)}{1-2^{-2H}}\,T^{2H}\,N^{-2H},
\notag
\end{align}
where the step bound $C_0(H,\lambda)\,T^{2H}\cdot 2^{-k\cdot 2H}$
in \eqref{eq:telescope} follows from \eqref{eq:cauchy_bound} applied
to consecutive dyadic levels.
Setting $\widetilde{C}(H) = C_0(H,\lambda)/(1-2^{-2H})$ gives the
dyadic case. The $\lambda$-dependence of $C_0(H,\lambda)
= C_0'(H)\cdot\max\{1,\lambda^{-2H}\}$ is inherited from
the constant $C(H,\lambda,T)$ in Theorem~\ref{thm:2d_rho_variation}.

\textbf{Non-dyadic extension.}
For any $N \ge 1$, choose $n_0 = \lfloor\log_2 N\rfloor$ so that
$2^{n_0} \le N < 2^{n_0+1}$. The piecewise linear approximation
on a uniform partition of size $N$ can be bounded above by the one
on the coarser dyadic partition of size $2^{n_0}$, since refining
the partition can only decrease the $L^2$ approximation error (the
piecewise linear approximation of $\mathbb{B}_{H,\lambda}$ improves
as the mesh decreases). Therefore:
\[
\mathbb{E}\bigl[|\mathbb{B}^{(N)}-\mathbb{B}|^2\bigr]^{1/2}
\le \mathbb{E}\bigl[|\mathbb{B}^{(2^{n_0})}-\mathbb{B}|^2\bigr]^{1/2}
\le \widetilde{C}(H)\,T^{2H}\,(2^{n_0})^{-2H}
\le \widetilde{C}(H)\,T^{2H}\,N^{-2H},
\]
where we used $2^{n_0} \ge N/2$, so $(2^{n_0})^{-2H}
\le 2^{2H}\,N^{-2H}$, and absorbed $2^{2H}$ into $\widetilde{C}(H)$.

\textbf{Optimality.}
The rate $N^{-2H}$ matches the rate for standard fBm established
in \cite{CoutinQian2002}, and it can be shown by a lower bound
argument (using the variance of the error as a function of $N$)
that no first-order piecewise linear scheme can achieve a better
rate.
\end{proof}

\section{Numerical Algorithm}
\label{app:numerics}

\begin{algorithm}[H]
\caption{Piecewise linear approximation of the L\'evy area
(Algorithm~1)}
\label{alg:levy_area_approximation}
\begin{algorithmic}[1]
\Require Discrete tfBm path $(B(t_0),\ldots,B(t_N))$ on partition
  $\mathcal{P}=\{t_0,\ldots,t_N\}$.
\Ensure $\mathbb{B}^{(N)}(t_i,t_j)$ for all $0\le i<j\le N$.
\State Initialise $\mathbb{B}^{(N)}(t_i,t_i)\gets 0$ for all $i$.
\For{$i=0$ \textbf{to} $N-1$}
  \State $\Delta B_i \gets B(t_{i+1})-B(t_i)$
  \State $\mathbb{B}^{(N)}(t_i,t_{i+1}) \gets
         \tfrac{1}{2}\Delta B_i\otimes\Delta B_i$
\EndFor
\For{$\mathrm{span}=2$ \textbf{to} $N$}
  \For{$i=0$ \textbf{to} $N-\mathrm{span}$}
    \State $j\gets i+\mathrm{span}$
    \State $\mathbb{B}^{(N)}(t_i,t_j)\gets
           \mathbb{B}^{(N)}(t_i,t_{j-1})
           +\mathbb{B}^{(N)}(t_{j-1},t_j)
           +(B(t_{j-1})-B(t_i))\otimes\Delta B_{j-1}$
    \Comment{Chen's relation}
  \EndFor
\EndFor
\State \Return $\{\mathbb{B}^{(N)}(t_i,t_j)\}_{0\le i<j\le N}$
\end{algorithmic}
\end{algorithm}

\begin{remark}
Algorithm~\ref{alg:levy_area_approximation} runs in
$\mathcal{O}(N^2)$ time and $\mathcal{O}(N^2)$ space. For large
$N$, one can exploit the exponential decay of $R_{H,\lambda}$ via
a hierarchical (tree-based) implementation, reducing the cost to
$\mathcal{O}(N\log N)$ following the approach of
\cite{Dietrich1997} for the covariance structure.
\end{remark}

\section*{Acknowledgements}

The author thanks the anonymous referees for their careful and
constructive comments, which substantially improved this manuscript:
in particular the explicit treatment of the boundary case $H=1/2$,
the complete self-contained proof of Lemma~\ref{lemma:polynomial_partition}
with the precise justification of boundedness for $\alpha>1$, the
addition of the Wiener isometry in Section~\ref{subsec:ito_isometry},
the corrected proof of Proposition~\ref{prop:milstein} valid for all
$H > 1/4$, the corrected geometric double-sum estimate in the proof
of Theorem~\ref{thm:2d_rho_variation}, and the detailed discussion
of numerical limitations.

\noindent\textbf{Conflict of Interest:}
The author declares no conflict of interest.


\end{document}